\documentclass[12pt,letterpaper]{amsart}

\usepackage{amssymb,latexsym,amsmath,amsthm,graphicx}
\numberwithin{equation}{section}
\newtheorem{thm}{Theorem}[section]
\newtheorem{lemma}[thm]{Lemma}

\newtheorem*{remark}{Remark}

\newcommand{\ol}{\overline}

\def \CC{\mathbb{C}}
\def \RR{\mathbb{R}}

\def \e{\varepsilon}

\begin{document}

\title[The mapping problem]{A solution to Sheil-Small's harmonic mapping problem for Jordan polygons}

\author[Daoud Bshouty]{Daoud Bshouty}
\author[Erik Lundberg]{Erik Lundberg}
\author[Allen Weitsman]{Allen Weitsman}
\address{Department of Mathematics\\ Technion\\ Haifa 32000\\
Israel
\\Email: daoud@tx.technion.ac.il}
\address{Department of Mathematics\\ Purdue University\\
West Lafayette, IN 47907-1395\\USA\\ Email: elundber@math.purdue.edu}
\address{Department of Mathematics\\ Purdue University\\
West Lafayette, IN 47907-1395\\USA\\ Email: weitsman@purdue.edu}



\begin{abstract}
The problem of mapping the interior of a Jordan polygon univalently by the Poisson integral of a step function
was posed by T. Sheil-Small (1989).
We describe a simple solution using ``ear clipping'' from computational geometry.
\end{abstract}
\maketitle
\section{Introduction}

In the subject of planar harmonic mappings, the mappings which arise as  Poisson integrals of step functions,
especially those which are univalent, play a prominent role (cf. \cite[pp.59-75] {Duren}).  
When the mapping is univalent, the image is a domain bounded by a Jordan polygon with its vertices at the steps of the boundary function. 

Mappings of this type  also appear when conformally parametrizing minimal graphs known as Jenkins-Serrin surfaces \cite{JS65}.  These are minimal graphs which take values $\pm \infty$ over the sides of a domain bounded 
by a Jordan polygon,  and if the parameter space is taken to be the unit disk $U$, then the first two coordinate functions of the parametrization give a univalent harmonic mapping which is given by the Poisson integral of a step function  \cite{BW}.

In 1989, T. Sheil-Small \cite{S-S1989} made a study of the mapping properties of Poisson integrals of step functions and  posed the following problem.  

\noindent{ \bf The mapping problem}.  {\it Given a  domain $D$ bounded by a Jordan polygon, does there exist a univalent harmonic mapping $f$, which is the Poisson integral of a step function, such that $f(U)=D$?}

There is a classical univalence criterion for harmonic mappings of $U$ onto a convex domain.  The problem
stated by T. Rad\'o in 1926 \cite{Rado26} and solved by H. Kneser \cite{Kneser26} the same year, shows that
for any homeomorphism of the unit circle $\partial U$ onto the boundary $\partial D$ of a convex domain $D$, the harmonic extension maps $U$ univalently onto $D$.  Later G. Choquet \cite{Choquet45} gave another proof 
which allowed the boundary function to be constant on arcs, and even to have jump discontinuities.  Thus, by
Choquet's theorem, the mapping problem has a positive solution when the polygon is convex.

This mapping problem has also been repeated in the book \cite[p. 402]{S-S2002} 
and more recently in the book \cite[p. 314]{MAA2012}).  
Also it was conjectured in \cite{S-S1989} that there would be polygons for which there is no such mapping.

In this paper we shall describe an algorithm which leads to a positive solution to the mapping problem.

\begin{thm}\label{thm:main}
Given any Jordan polygon $\Pi$ bounding a domain $D$, there exists a step function $f (e^{it})$ on $\partial U$ whose harmonic extension gives a univalent harmonic mapping of $U$ onto $D$.
\end{thm}

In order to state the problem precisely, consider the polygon $\Pi = [c_1,c_2,..,c_n, c_1]$
as a positively oriented Jordan curve with distinct vertices $c_k$
bounding a domain $D$.
For a sequence of intervals
$0 = t_0 < t_1 < .. < t_n = t_0 + 2 \pi$,
consider 
the Poisson integral $f(z)$ of the (complex) step function (also denoted by $f$)
\begin{equation}\label{f(x)1}
f(e^{it}) = c_k \quad (t_{k-1} < t < t_k).
\end{equation}
The problem then comes down to showing that there is a choice of $t_0, t_1,.....,t_n$ for 
which the Poisson extension is univalent.

We prove this theorem in Section 2.
In Section 3, we describe an estimate for harmonic measure in a half-plane which can 
be useful in determining the univalence across the sides of a polygon as it arises as the image of
the Poisson integral of a step function. 

\section{A solution to the mapping problem (proof of Theorem \ref{thm:main})}

We follow the notation in \cite{S-S1989}.

With $f(z) = h(z) + \ol{g(z)}$ represented by analytic functions $h$ and $g$,
it is enough to show, by Theorem 5 in \cite{S-S1989}, 
that the zeros of 
$$h'(z) = \sum_{k=1}^n \frac{\alpha_k}{z-\zeta_k}$$
are outside $U$ the unit disk,
where $\alpha_k = \frac{1}{2 \pi i} (c_k - c_{k+1} )$ for $k< n$,
and $\alpha_n = \frac{1}{2 \pi i} (c_n - c_{1} ),$
and $\zeta_k = e^{it_k}$, for $k=1,..,n$.

We give a proof by induction on the number of vertices.

{\bf Induction statement:} Given any Jordan $n$-gon, there exists a sequence of intervals 
such that the zeros of $h'(z)$ are in $\CC \setminus \overline{U}$.

The ``base case'' $n=3$ is a triangle and the statement follows from the Rado-Kneser-Choquet Theorem.

Inductive step: Suppose the Induction Statement is true up to some $n$.
Consider a Jordan $n+1$-gon, $ \Pi = [c_1,c_2,..,c_{n+1}, c_1]$.
We follow the triangulation algorithm known as ``ear clipping''.
Namely, find a vertex of $\Pi$, without loss of generality assume the vertex is $c_{n+1}$, 
such that removing it results in a Jordan polygon with $n$ vertices.
Such a vertex is called an ``ear'' in computational geometry,
and it is well-known that every Jordan polygon has at least two ears \cite{Meisters75}.

{\bf Lemma A.}(The ``two ears'' theorem)
{\it Any Jordan polygon has at least two ``ears''.}

Consider the $n$-gon $\Pi' = [c_1,c_2,..,c_n, c_1]$.
By the induction statement, there is a choice of intervals 
$$0 = t_0 < t_1 < .. < t_n = t_0 + 2 \pi$$
so that the Poisson integral of the corresponding step function $f(z)$ is univalent, 
and the zeros of $h'(z)$ are outside the closed unit disk $\ol{U}$.

We construct a map $f_{\e}(z)$ to $\Pi$ 
by making a new choice of intervals 
$$0 = \tau_0 < \tau_1 < .. < \tau_n <\tau_{n+1} = \tau_0 + 2\pi .$$
Namely, we take $\tau_k = t_k$ for $k=0,1,..,n-1$ and $\tau_n = 2\pi - \e$.
Then $f_{\e}(z)$ is taken to be the Poisson integral of the step function
$$f_\e(e^{it}) = c_k \quad (\tau_{k-1} < t < \tau_k).$$
In comparison with the choice of intervals used for the map $f(z)$,
this slightly alters the interval corresponding to $c_n$
and introduces a new interval $(\tau_n,\tau_{n+1})$ of size $\e$ corresponding to the ear $c_{n+1}$ (see Figure \ref{fig:ear}
\begin{figure}[h]
    \begin{center}
   \includegraphics[scale=.3]{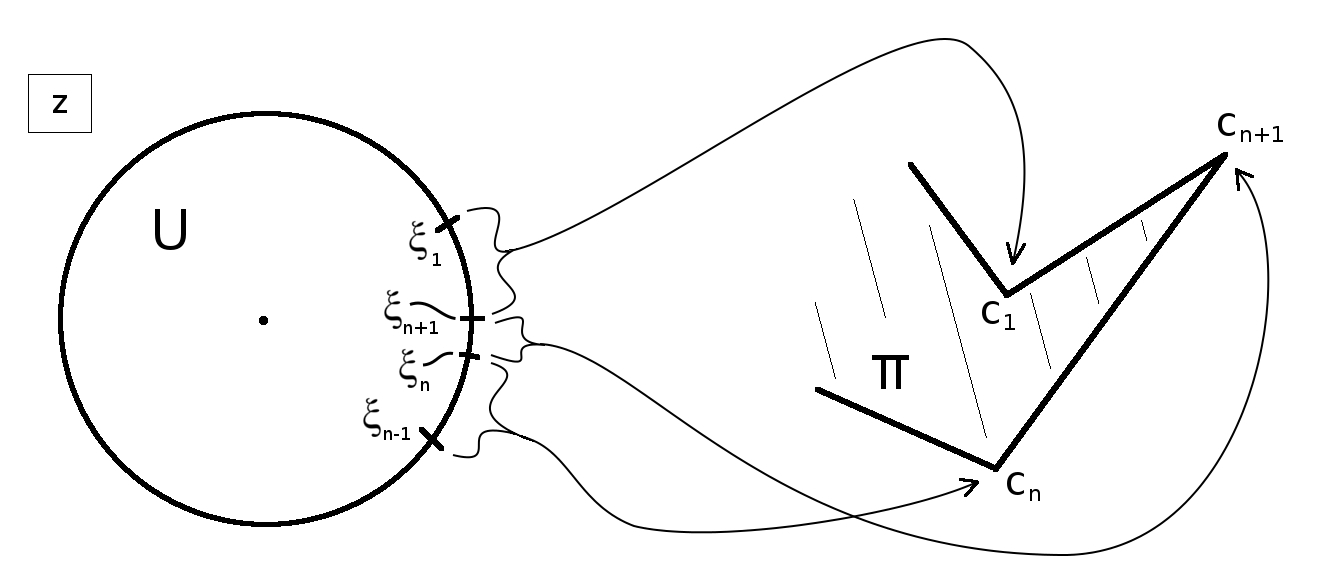}
    \end{center}
    \caption{Illustration of the map $f_\e$ in the vicinity of the ear $c_{n+1}$.}
    \label{fig:ear}
\end{figure}

With these choices and using the notation $\beta_k = 
\displaystyle{\frac{1}{2 \pi i} }(c_k - c_{k+1})$ for $k=1,..,n$,
$\beta_{n+1} = \displaystyle{\frac{1}{2 \pi i}} (c_{n+1} - c_{1})$, 
and $\xi_k = e^{i\tau_k}$ for $k=1,..,n+1$, we have 
\begin{equation}\label{eq:heps}
h_{\e}'(z) = \sum_{k=1}^{n+1} \frac{\beta_k}{z- \xi_k},
\end{equation}
where $h_\e (z)$ is the analytic part of $f_\e (z)$.

\noindent {\bf Claim 1:}
As $\e \rightarrow 0$, $h_{\e}'$ approximates $h'$ uniformly outside any neighborhood of the point $\xi_{n+1}=1$.

To verify Claim 1, we note that the first $n-1$ terms in the sum (\ref{eq:heps})
are the same as the first $n-1$ terms in
$$h'(z) = \sum_{k=1}^n \frac{\alpha_k}{z-\zeta_k}.$$
Thus, in order to prove  Claim 1 we only need to check that as $\e \rightarrow 0$ the last two terms in $h'_{\e}$
$$ \frac{\beta_n}{z- \xi_n} + \frac{\beta_{n+1}}{z- \xi_{n+1}}$$
converge uniformly to the last term in $h'$
$$ \frac{\alpha_n}{z-\zeta_n} = \frac{c_n - c_1}{z-1},$$
which follows from 
$$\frac{\beta_n}{z- \xi_n} + \frac{\beta_{n+1}}{z- \xi_{n+1}} = 
\frac{c_n-c_{n+1}}{z- e^{-i\e}} + \frac{c_{n+1} - c_1}{z- 1}.$$

By  Claim 1 and Hurwitz's theorem, $h_{\e}'$ has a zero near each of the zeros of $h'$.
For $\e>0$ sufficiently small, this places at least $n-2$ (counting multiplicities) of the $n-1$ zeros of $h_{\e}'$
outside $\ol{U}$ (not counting $\infty$ which is a zero of multiplicity two).

It remains to show that the final zero of $h'_\e$ is also outside $\ol{U}$.
Estimating the location of this zero is slightly complicated by the fact that it converges to $\xi_{n+1} = 1$,
the point where two poles are merging as $\e \rightarrow 0$.
Thus, we use a renormalization.

Let us write
$z = \e w + \xi_{n+1} = \e w + 1 $, and 
$$H_{\e}(w) := h_{\e}'( \e w + 1).$$
By the above, $H_{\e}(w)$ has $n-2$ zeros (counting multiplicities) converging to $\infty$ as $\e \rightarrow 0$.
The remaining zero converges to a finite point $w = w_0$ in the right half-plane as follows from the next claim.

\noindent {\bf Claim 2:} $H_{\e}(w)$ has a zero at $w = w_\e$ such that 
$$w_\e \rightarrow w_0 := -\frac{c_{n+1} - c_1}{c_n - c_{1}} i, \quad \text{as } \e \rightarrow 0 .$$

Before proving  Claim 2 let us see how it establishes the result.
It follows from the fact that $[c_n,c_{n+1},c_1]$ is an \emph{outside} corner
that the argument of $\displaystyle{\frac{c_{n+1} - c_1}{c_n - c_1}}$ is strictly between zero and $\pi$.
Thus, $-\displaystyle{\frac{c_{n+1} - c_1}{c_n - c_1}} i$ is in the right half-plane,
and, for $\e$ sufficiently small, $w_\e$ is also in the right half-plane.
This places $\e w_\e + 1$ (the remaining zero of $h_{\e}'$) outside of $\ol{U}$,
and this completes the inductive step.  

It remains to prove Claim 2.

\begin{proof}[Proof of Claim 2]
 
We have 
\begin{equation}\label{eq:epsH1}
\e H_{\e}(w) = \sum_{k=1}^{n+1} \frac{ \beta_k}{w + (\xi_{n+1} - \xi_k)/ \e}.
\end{equation}
Choose a disk $V$ centered at the point
$$w_0 := -\frac{c_{n+1} - c_1}{c_n - c_{1}} i$$ 
such that $\ol{V}$ omits each of the points $w=0$ and $w=-i$.
As $\e \rightarrow 0$, the first $n-1$ terms in (\ref{eq:epsH1}) converge uniformly to zero in $V$ while the final
terms
$$\frac{ \beta_n}{w + (\xi_{n+1} - \xi_n)/ \e } 
+ \frac{ \beta_{n+1}}{w}$$
converge uniformly to
\begin{equation}\label{eq:limit}
\frac{ \beta_n}{w + i } 
+ \frac{ \beta_{n+1}}{w},
\end{equation}
where we have used 
$$\frac{\xi_{n+1} - \xi_n}{\e} = \frac{1 - e^{-i \e}}{\e} = \frac{1 - (1 - i \e + O(\e^2) )}{\e} = i + O(\e).$$
This last expression (\ref{eq:limit}) has a single zero in $V$,
namely at 
$$w_0 = - \frac{\beta_{n+1}}{\beta_n+\beta_{n+1}}i = -\frac{c_{n+1} - c_1}{c_n - c_1} i.$$
By Hurwitz's Theorem (\ref{eq:epsH1}) has exactly one zero in $V$, and this zero converges to $w_0$.

\end{proof}

\section{The law of sines lemma}

In this section, we replace the unit disk with the upper half-plane $H$, and consider the harmonic measure
$\omega_I$ in  $H$ of an interval $I$ on the real axis.
The asymptotic behavior of $\omega_I(z)$ as $z \in H$ approaches a point on the real axis can be used
to detect possible folding near an edge of the image polygon $\Pi$.
Lemma \ref{LOS} below was an initial guide for Theorem \ref{thm:main},
although in the end we did not need it.  However, it seems that it may be of independent interest.

Recall that for $z \in H$, the harmonic measure $\omega_I(z)$ of an interval $I \subset \RR$ equals
the angle between the two segments joining $z$ to each of the endpoints of $I$.

\begin{lemma}[Law of sines lemma]
\label{LOS}

Suppose $x_0 < x_1 < x_2$ are points on the real axis and the intervals $[x_0,x_1]$ and $[x_1,x_2]$ 
have length $A$ and $B$ respectively. 
Suppose $z \in H$ approaches $x_0$ along a segment.  
Let $y$ be the imaginary part of $z$, and let $\omega(z)$ denote the harmonic measure of $[x_1,x_2]$.
Then, 
$$ \frac{\omega(z)}{y}  \rightarrow \frac{B}{A^2 + AB} \quad (\text{as } z \rightarrow x_0).$$
In particular, this limit is independent of the angle of approach of $z \rightarrow x_0$.
\end{lemma}

\begin{proof}[Proof of Lemma]

Let $x$ be the real part of $z$ and let $A' = x_1 - x$.

By the law of sines,
$$ \frac{\sin \theta}{B} = \frac{\sin \psi}{C} = \frac{y / \sqrt{y^2+(A'+B)^2}}{\sqrt{y^2+A'^2}} ,$$
where $\psi$ is the angle of the corner $[x_0,x_2,z]$, and $C$ is the length of the segment $[x_1,z]$.

Rearranging,
$$ \frac{\sin \theta}{y} = \frac{B}{\sqrt{y^2 + (A' + B)^2} \sqrt{y^2 + A'^2}} .$$

Letting $z \rightarrow x_0$, we have $A' \rightarrow A$, and $y \rightarrow 0$.
Thus,
$$ \frac{\sin \theta}{y} \rightarrow \frac{B}{(A + B)A} .$$
Since $$\sin \theta = \theta +  O(\theta^3),$$
we have
$$ \frac{\omega(z)}{y} \rightarrow \frac{B}{A^2 + AB},$$
as $z \rightarrow x_0$.
\end{proof}
\begin{figure}[h]
    \begin{center}
    \includegraphics[scale=.3]{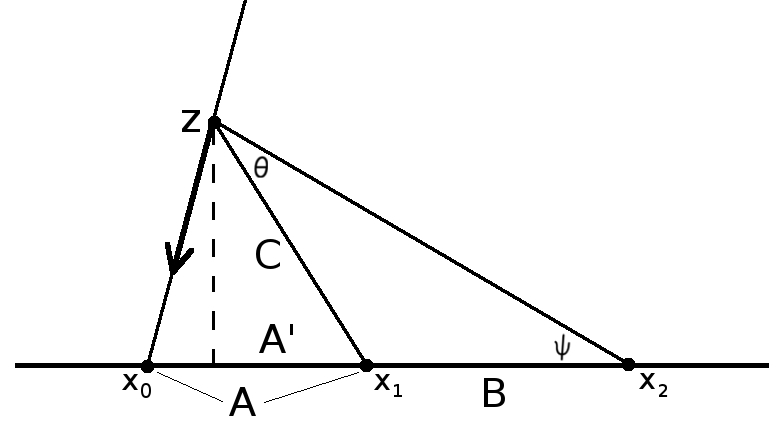}
    \end{center}
    \caption{$z$ approaches $x_0$ along a fixed angle.}
    \label{fig:LOS}
\end{figure}

In order to see how this can be useful, suppose that $\Pi=[c_1,c_2,.....,c_n,c_1]$ is a Jordan polygon
and $f$ is the harmonic extension of the corresponding step function as before, but now composed with
a M\"obius transformation so it is defined in $H$.  Then, corresponding to (\ref{f(x)1}) we have points
$$\zeta_1 < \zeta_2 < .. < \zeta_n$$ 
on the real axis, 
$$f(x) = c_k \quad (\zeta_{k} < x < \zeta_{k+1}),$$

for $k=1,....,n-1$, and $f(x) = c_n$ for the interval $\{x<\zeta_1\} \cup \{x > \zeta_n\}$ containing infinity.

Let $\omega_k(z)$ be the harmonic measure of the interval $[\zeta_{k},\zeta_{k+1}]$ with respect to $z$.
The map $f$ can be expressed simply as a linear combination of the vertices $c_k$ weighted by the harmonic measure $\omega_k(z)$ of the interval that is mapped to $c_k$:
\begin{equation}\label{f(x)}
 f(z) = c_1 \omega_1(z) + c_2 \omega_2(z) + .. + c_n \omega_n(z) .
\end{equation}

As noted in the introduction, if the map $f$ fails to be univalent, then there must be folding over the boundary,
so it is natural to consider the local behavior of $f$ near an edge.
Fix $m$ and let $z \rightarrow \zeta_m$.
Then $f(z)$ approaches a value on the edge $[c_{m-1},c_{m}]$.
Applying the lemma, we obtain an approximation for $\omega_k(z)$ in terms of the lengths $\ell_j$ of the intervals $[\zeta_{j},\zeta_{j+1}]$.
Namely, when $m<k<n$,
$$\omega_k(z) / y \approx \frac{\ell_k}{(\ell_m + \ell_{m+1} + .. + \ell_{k-1})^2 + (\ell_m + \ell_{m+1} + .. + \ell_{k-1})\ell_k}.$$
We obtain a similar expression for $\omega_k(z)$ when $k<m$, and when $k=n$ (corresponding to the infinite interval) we have
$$\omega_n(z) / y \approx \frac{1}{\ell_m + \ell_{m+1} + .. + \ell_{n-1}}.$$

Guided by these approximations, one may quantify in terms of the relative lengths of the intervals,
the contributions of the individual terms in (\ref{f(x)}).  For the mapping problem of the current 
paper,
 one may choose the lengths $\ell_k$ in order to prevent folding near some edge $[c_{m-1},c_m]$,
and in order to simultaneously prevent folding over all edges, the lengths $\ell_k$ must satisfy a system of inequalities.

\bibliographystyle{amsplain}

\end{document}